\documentclass{amsart}

\usepackage{amssymb}

\usepackage[pagewise, displaymath, mathlines]{lineno}

\makeatletter
\let\my@abstract=\relax
\def\abstract#1{%
  \def\my@abstract{%
    \normalfont\Small
    \list{}{\labelwidth\z@
      \leftmargin3pc \rightmargin\leftmargin
      \listparindent\normalparindent \itemindent\z@
      \parsep\z@ \@plus\p@
      \let\fullwidthdisplay\relax
    }%
    \item[\hskip\labelsep\scshape\abstractname.]%
    #1
  \endlist}}
\def\@setabstracta{%
  \ifx\my@abstract\relax
  \else
    \skip@20\p@ \advance\skip@-\lastskip
    \advance\skip@-\baselineskip \vskip\skip@
  \my@abstract
    \prevdepth\z@ 
  \fi
}
\makeatother

\newtheorem{theorem}{Theorem}[section]
\newtheorem{lemma}[theorem]{Lemma}
\newtheorem{cor}[theorem]{Corollary}

\theoremstyle{definition}

\newtheorem{rem}[theorem]{Remark}

\newtheorem{prop}[theorem]{Proposition}
\theoremstyle{remark}

\numberwithin{equation}{section}

\input{xy}
\xyoption{all}

\begin{document}


\title[Witt equivalence of conics]{Witt equivalence of function fields of conics}

\author{Pawe\l \ G\l adki}

\address{Institute of Mathematics,
University of Silesia, \newline \indent
ul. Bankowa 14, 40-007 Katowice, Poland \newline \indent
\and \newline \indent
Department of Computer Science,
AGH University of Science and Technology, \newline \indent
al. Mickiewicza 30, 30-059 Krak\'ow, Poland}
\email{pawel.gladki@us.edu.pl}

\author{\fbox{Murray Marshall}}

\thanks{Murray Marshall passed away in May 2015. Our community lost a brilliant mathematician and a wonderful man. We sorely miss him.}



\abstract{Two fields are Witt equivalent if, roughly speaking, they have the same quadratic form theory. Formally, that is to say that their Witt rings of symmetric bilinear forms are isomorphic. This equivalence is well understood only in a few rather specific classes of fields. Two such classes, namely function fields over global fields and function fields of curves over local fields, were investigated by the authors in their earlier works \cite{gm2016} and \cite{gm2016a}. In the present work, which can be viewed as a sequel to the earlier papers, we discuss the previously obtained results in the specific case of function fields of conic sections, and apply them to provide a few theorems of a somewhat quantitive flavour shedding some light on the question of numbers of Witt non-equivalent classes of such fields.}

\maketitle

\section{Introduction}

One of the classical problems in bilinear algebra is to classify fields with respect to Witt equivalence, that is equivalence defined by isomorphism of their Witt rings of symmetric bilinear forms, which also includes fields of characteristic two. This problem is, in fact, manageable only when restricted to some specific classes of fields, which include trivial examples of quadratically closed fields, real closed fields, and finite fields, the case of local fields (\cite{l}), global fields (\cite{Petal}, \cite{Sz1}, \cite{szym97}), function fields in one variable over algebraically closed fields of characteristic $\neq 2$, and function fields in one variable over real closed fields (\cite{gh}, \cite{kop2}).

The authors of the present paper attempted to add two more classes of fields to this list, and investigated function fields over global fields \cite{gm2016} and function fields of curves over local fields \cite{gm2016a}, and managed to show that Witt equivalence of two function fields over global fields induces in a canonical way a bijection $v \leftrightarrow w$ between Abhyankar valuations $v$ of $K$ having residue field not finite of characteristic 2 and Abhyankar valuations $w$ of $L$ having residue field not finite of characteristic 2 (\cite[Theorem 7.5]{gm2016}). Subsequently, a variant of this theorem has been also established in the local case (\cite[Theorem 3.5]{gm2016a}). Numerous corollaries providing some insight into the question of how Witt equivalence in these cases is behaved have been also drawn.

In the present paper we apply these results to take a closer look at the question of the number of Witt non-equivalent classes of function fields of conics, and provide some enumerative results in the case of conics defined over certain number fields. Some of them generalize in a certain way to the case of conics defined over arbitrary local fields. The main new results of the paper are found in Section 4. Throughout the entire exposition the authors use the language of hyperfields, which seem to provide a natural and convenient language to study Witt equivalence. We recall basic terminology and establish fundamental connections between hyperfields, valuations and quadratic forms in Section 2. All of this is a summary of Section 2 in \cite{gm2016a}, which, in turn, is a summary of Sections 2--6 in \cite{gm2016}, and the reader more interested in all the technicalities is kindly referred to consult author's first paper \cite{gm2016} in the sequel. In Section 3 the authors prove a few additional facts on function fields of conics, and cite some old propositions that go back to Ernst Witt. The authors would like to believe that their results can be thought of as extensions of these beautiful, classical theorems by old masters.

\section{Hyperfields, valuations and Witt equivalence}

Hyperfields seem to provide a convenient and very natural way to describe Witt equivalence. In what follows we shall review the basic concepts and definitions used later in the paper. By a \it hyperfield \rm we shall understand a system $(H, +, \cdot, -, 0, 1)$, where $H$ is a set, $+$ is a multivalued binary operation on $H$, i.e., a function from $H\times H$ to the set of all subsets of $H$, $\cdot$ is a binary operation on $H$, $- : H \rightarrow H$ is a function, and $0,1$ are elements of $H$ such that

\begin{enumerate}
\item[(I)] $(H,+,-,0)$ is a {\it canonical hypergroup}, i.e. for all $a,b,c \in H$,
\begin{enumerate}
\item[(1)] $c\in a+b$ $\Rightarrow$ $a \in c+(-b)$,
\item[(2)] $a \in b+0$ iff $a=b$,
\item[(3)]$(a+b)+c = a+(b+c)$,
\item[(4)] $a+b= b+a$;
\end{enumerate}
\item[(II)] $(H,\cdot, 1)$ is a commutative monoid, i.e. for all $a,b,c \in H$, 
\begin{enumerate}
\item[(1)] $(ab)c=a(bc)$, 
\item[(2)]$ab=ba$,
\item[(3)] $a1=a$;
\end{enumerate}
\item[(III)] $a0 = 0$ for all $a\in H$;
\item[(IV)] $a(b+c) \subseteq ab+ac$;
\item[(V)] $1\ne 0$ and every non-zero element has a multiplicative inverse.
\end{enumerate}

Hyperfields form a category with morphisms from $H_1$ to $H_2$, where $H_1$, $H_2$ are hyperfields, defined to be functions $\alpha : H_1 \rightarrow H_2$ which satisfy $\alpha(a+b) \subseteq \alpha(a)+\alpha(b)$, $\alpha(ab) = \alpha(a)\alpha(b)$, $\alpha(-a)=-\alpha(a)$, $\alpha(0)=0$, $\alpha(1)=1$. For a subgroup $T$ of $H^*$, where $H$ is a hyperfield, denote by $H/_m T$ the set of equivalence classes with respect to the equivalence relation $\sim$ on $H$ defined by 
$$a\sim b \mbox{ if and only if } as = bt \mbox{ for some } s,t\in T.$$ 
The operations on $H/_m T$ are the obvious ones induced by the corresponding operations on $H$. Denote by $\overline{a}$ the equivalence class of $a$. Multiplication is defined in a natural way, and addition is set as follows:
$$\overline{a} \in \overline{b}+\overline{c} \mbox{ if and only if } as \in bt+cu \mbox{ for some }s,t,u \in T.$$
$(H/_m T, +,\cdot, -, 0,1)$ is then a hyperfield that we shall refer to as {\em quotient hyperfield}.  For a hyperfield $H = (H, +,\cdot, -, 0,1)$ the \it prime addition \rm on $H$ is defined by
\begin{linenomath}\[a+'b = \begin{cases} a+b, &\text{ if one of } a,b \text{ is zero }, \\ a+b\cup \{ a,b\}, &\text{ if } a \ne 0, \ b\ne 0,\ b \ne -a, \\ H, &\text{ if } a \ne 0, \ b\ne 0,\ b = -a. \end{cases}\]\end{linenomath}
For any hyperfield $H:=(H, +,\cdot, -, 0,1)$, $H':=(H, +',\cdot, -, 0,1)$ is also a hyperfield \cite[Proposition 2.1]{gm2016}. We shall call $H'$ the \it prime \rm of the hyperfield $H$. The motivation for this definition comes from the following discussion: let $K$ be a field and define the \it quadratic hyperfield \rm of $K$, denoted $Q(K)$, to be the prime of the hyperfield $K/_m K^{*2}$. Now let $W(K)$ be the Witt ring of non-degenerate symmetric bilinear forms over $K$; see \cite{l}, \cite{M1} or \cite{w} for the definition in case $\operatorname{char}(K) \ne 2$ and \cite{KR}, \cite{krw} or \cite{mh} for the definition in the general case. Recall that a (non-degenerate diagonal) {\it binary form} over $K$ is an ordered pair $\langle \overline{a},\overline{b}\rangle$, $\overline{a}, \overline{b} \in K^*/K^{*2}$, and its \it value set \rm, denoted by  $D_K\langle \overline{a},\overline{b} \rangle$, is the set of non-zero elements of $\overline{a}+\overline{b}$.  Now, a hyperfield isomorphism $\alpha$ between two quadratic hyperfields $Q(K)$ and $Q(L)$, where $K,L$ are fields, can be viewed as a group isomorphism $\alpha : K^*/K^{*2} \rightarrow L^*/L^{*2}$ such that $\alpha(-\overline{1}) = -\overline{1}$ and \begin{linenomath}\[\alpha(D_K\langle \overline{a},\overline{b}\rangle) = D_L\langle \alpha(\overline{a}),\alpha(\overline{b})\rangle \text{ for all } \overline{a},\overline{b} \in K^*/K^{*2}.\]\end{linenomath} Combining the results in  \cite{bm}, \cite{h}, and \cite{M1} one gets that two fields $K$ and $L$ are Witt equivalent, denoted $K\sim L$, iff $Q(K)$ and $Q(L)$ are isomorphic as hyperfields. Moreover, a morphism $\iota: H_1 \rightarrow H_2$ between two hyperfields $H_1$ and $H_2$ induces a morphism $\overline{\iota} : H_1/_m \Delta \rightarrow H_2$ where $\Delta: = \{ x\in H_1^* : \iota(x)=1\}$. The morphism $\iota$ is said to be a \it quotient morphism \rm if $\overline{\iota}$ is an isomorphism, or, equivalently, if $\iota$ is surjective, and 
$$\iota(c) \in \iota(a)+\iota(b) \mbox{ if and only if } cs\in at+bu \mbox{ for some } s,t,u \in \Delta.$$  
A morphism $\iota : H_1 \rightarrow H_2$ is said to be a \it group extension \rm if $\iota$ is injective,
every $x\in H_2^* \backslash \iota(H_1^*)$ is \it rigid \rm, meaning $1+x \subseteq \{ 1,x\}$, and 
$$\iota(1+y)=1+\iota(y), \mbox{ for all } y \in H_1, y\ne -1.$$

For a field $K$ we adopt the standard notation from valuation theory: if $v$ is a valuation on $K$, $\Gamma_v$ denotes the value group, $A_v$ the valuation ring, $M_v$ the maximal ideal, $U_v$ the unit group, $K_v$ the residue field, and $\pi = \pi_v: A_v \rightarrow K_v$ the canonical homomorphism, i.e., $\pi(a) = a+M_v$. $v$ is \it discrete rank one \rm if $\Gamma_v = \mathbb{Z}$. Denote $T= (1+M_v)K^{*2}$. Consider the canonical group isomorphism $\alpha : U_vK^{*2}/(1+M_v)K^{*2} \rightarrow K_v^*/K_v^{*2}$ induced by 
$$x \in U_v \mapsto \pi(x) \in K_v^*.$$ 
and define $\iota : Q(K_v) \rightarrow K/_m T$ by $\iota(0) = 0$ and $\iota(a) = \alpha^{-1}(a)$ for $a \in K_v^*/K_v^{*2}$. If $v$ is non-trivial, then the map $Q(K) \rightarrow K/_mT$ defined by $\overline{x} \mapsto xT$  is a quotient morphism and $\iota$ is a group extension. The cokernel of the group embedding $\alpha^{-1} : K_v^*/K_v^{*2} \rightarrow K^*/T$ is equal to $K^*/U_vK^{*2} \cong \Gamma_v/2\Gamma_v$. We reflect this by calling $K/_mT$ a \it group extension of $Q(K_v)$ by the group \rm $\Gamma_v/2\Gamma_v$. If $v$ is non-trivial and $\operatorname{char}(K_v) \ne 2$, then $K/_m T$ can be naturally identified with $Q(\tilde{K}_v)$, where $\tilde{K}_v$ denotes the henselization of $(K,v)$. If $v,v'$ are valuations on $K$ with $v \preceq v'$, i.e. such that $v'$ is a coarsening of $v$, meaning $A_v \subseteq A_{v'}$, then $M_{v'} \subseteq M_v$, and, consequently, $(1+M_{v'})K^{*2} \subseteq (1+M_v)K^{*2}$. If we denote by $\overline{v}$ the valuation on $K_{v'}$ induced by $v$, that is
$$\overline{v}(\pi_{v'}(a)) = v(a), \mbox{ for } a\in U_{v'},$$
then the valuations $\overline{v}$ and $v$ have the same residue field. If $v$ and $v'$ are non-trivial and $v'$ is a proper coarsening of $v$, meaning $A_v \subsetneq A_{v'}$, then $K/_m (1+M_v)K^{*2}$ is a group extension of the hyperfield $K_{v'}/_m(1+M_{\overline{v}})K_{v'}^{*2}$ and the following diagram is commutative:
\begin{linenomath}\[\xymatrix{
Q(K) \ar[r] & K/_m(1+M_{v'})K^{*2} \ar[r] & K/_m(1+M_v)K^{*2} \\
& Q(K_{v'}) \ar[u] \ar[r] & K_{v'}/_m (1+M_{\overline{v}})K_{v'}^{*2} \ar[u]\\
& & Q(K_v) \ar[u]
}\]\end{linenomath}
For a subgroup $T$ of $K^*$ we say that $x\in K^*$ is \it $T$-rigid \rm if $T+Tx \subseteq T\cup Tx$, and denoting by 
\begin{linenomath}\[B(T) := \{ x \in K^*: \text{ either } x \text{ or } -x \text{ is not } T\text{-rigid}\}\]\end{linenomath}
we will refer to the elements of $B(T)$ as to the \it $T$-basic \rm elements.
If $x\in K^*$ is $T$-rigid and
$y=tx$, for some $t\in T$, then $y$ is $T$-rigid, so that $B(T)$ is a union of cosets of $T$.
If $\pm T = B(T)$, and either $-1\in T$ or $T$ is additively closed, we shall say that the subgroup $T$ is {\it exceptional.}
If $H \subseteq K^*$ is a subgroup containing $B(T)$, then there exists a subgroup $\hat{H}$ of $K^*$ such that $H \subseteq \hat{H}$ and $(\hat{H}:H) \le 2$, and a valuation $v$ of $K$ such that $1+M_v \subseteq T$ and $U_v \subseteq \hat{H}$. Moreover, $\hat{H}$ can be taken to be simply $H$, unless $T$ is exceptional  \cite[Theorem 2.16]{aej}. $B(K^{*2})$ is a subgroup of $K^*$, and in the case when $T = (1+M_v)K^{*2}$, for some non-trivial valuation $v$ of $K$, $B(T) \subseteq U_vK^{*2}$ and \begin{linenomath}\[B(T) = \{ x\in K^* : \overline{x} = \iota(\overline{y}) \text{ for some } y \in B(K_v^{*2})\},\]\end{linenomath} where $\iota : Q(K_v) \hookrightarrow K/_mT$ is the morphism described above.  $B(T)$ is a group and the group isomorphism $\iota : K_v^*/K_v^{*2} \rightarrow U_vK^{*2}/T$ induces a group isomorphism $B(K_v^{*2})/K_v^{*2} \rightarrow B(T)/T$. $T$ is exceptional if and only if $K_v^{*2}$ is exceptional. We will make frequent use of the following result:

\begin{theorem}[{\cite[Theorem 5.3]{gm2016}}] \label{main lemma} Suppose $K$, $L$ are fields, $\alpha : Q(K) \rightarrow Q(L)$ is a hyperfield isomorphism and $v$ is a valuation on $K$ such that $\Gamma_v$ is finitely generated as an abelian group. Suppose either
(i) the basic part of $(1+M_v)K^{*2}$ is $U_vK^{*2}$ and $(1+M_v)K^{*2}$ is unexceptional, or
(ii) the basic part of $(1+M_v)K^{*2}$ is $(1+M_v)K^{*2}$ and $(1+M_v)K^{*2}$ has index 2 in $U_vK^{*2}$.
Then there exists a valuation $w$ on $L$ such that the image of $(1+M_v)K^{*2}/K^{*2}$ under $\alpha$ is $(1+M_w)L^{*2}/L^{*2}$ and $(L^*:U_wL^{*2}) \ge (K^*:U_vK^{*2})$. If (i) holds, then
the image of $U_vK^{*2}/K^{*2}$ under $\alpha$ is $U_wL^{*2}/L^{*2}$.
\end{theorem}

If $K$, $L$ are fields, $v$ and $w$ are non-trivial, and $\alpha : Q(K) \rightarrow Q(L)$ is a hyperfield isomorphism such that the image of $(1+M_v)K^{*2}/K^{*2}$ under $\alpha$ is $(1+M_w)L^{*2}/L^{*2}$, then $\alpha$ induces a hyperfield isomorphism $K/_m (1+M_v)K^{*2} \rightarrow L/_m (1+M_w)L^{*2}$ such that the diagram
\begin{linenomath}\begin{equation}\label{d1} \xymatrix{
Q(K) \ar[r] \ar[d] & Q(L) \ar[d] \\
K/_m(1+M_v)K^{*2} \ar[r] & L/_m(1+M_w)L^{*2}}\end{equation}\end{linenomath}
commutes. If, in addition, the image of $U_vK^{*2}/K^{*2}$ under $\alpha$ is $U_wL^{*2}/L^{*2}$, then $\alpha$ induces a hyperfield isomorphism $Q(K_v)\rightarrow Q(L_w)$ and a group isomorphism $\Gamma_v/2\Gamma_v \rightarrow \Gamma_w/2\Gamma_w$ such that the diagrams
\begin{linenomath}\begin{equation}\label{d2} \xymatrix{
K/_m(1+M_v)K^{*2} \ar[r] & L/_m(1+M_w)L^{*2} \\
 Q(K_v) \ar[u] \ar[r]  & Q(L_w) \ar[u]
}\end{equation}\end{linenomath}
and
\begin{linenomath}\begin{equation}\label{d3} \xymatrix{
Q(K)^* \ar[r] \ar[d] & Q(L)^* \ar[d] \\
\Gamma_v/2\Gamma_v \ar[r] & \Gamma_w/2\Gamma_w
}\end{equation}\end{linenomath}
both commute. 

Recall that the \it nominal transcendence degree \rm of $K$ is defined to be
\begin{linenomath}\[\operatorname{ntd}(K) := \begin{cases} \operatorname{trdeg}(K:\mathbb{Q}) &\text{ if } \operatorname{char}(K) = 0 \\ \operatorname{trdeg}(K:\mathbb{F}_p)-1 &\text{ if } \operatorname{char}(K) = p\ne 0\end{cases}.\]\end{linenomath}
For an abelian group $\Gamma$, its \it rational rank \rm of $\Gamma$, denoted $\operatorname{rk}_{\mathbb{Q}}(\Gamma)$, is defined to be the dimension of the $\mathbb{Q}$-vector space $\Gamma \otimes_{\mathbb{Z}} \mathbb{Q}$. If $K$ is a function field over $k$ and $v$ is a valuation on $K$, the \it Abhyankar inequality \rm asserts that \begin{linenomath}\[\operatorname{trdeg}(K:k) \ge \operatorname{rk}_{\mathbb{Q}}(\Gamma_v/\Gamma_{v|k}) + \operatorname{trdeg}(K_v: k_{v|k}),\]\end{linenomath} where $v|k$ denotes the restriction of $v$ to $k$.  The valuation $v$ is \it Abhyankar \rm (relative to $k$)  if \begin{linenomath}\[\operatorname{trdeg}(K:k) = \operatorname{rk}_{\mathbb{Q}}(\Gamma_v/\Gamma_{v|k}) + \operatorname{trdeg}(K_v: k_{v|k}).\]\end{linenomath}
In this case it is known that $\Gamma_v/\Gamma_{v|k}$ is finitely generated and $K_v$ is a function field over $k_{v|k}$.

\section{function fields of conics}

Let $k$ be a field of characteristic $\ne 2$.

\begin{prop} $\frac{k[x,y]}{(ax^2+by^2-1)}$ is a principal ideal domain, for each $a,b \in k^*$.
\end{prop}

\begin{proof} This is well known, see \cite{gm}.
\end{proof}

Set $k_{a,b} := \operatorname{qf}\frac{k[x,y]}{(ax^2+by^2-1)}$, the quotient field of $\frac{k[x,y]}{(ax^2+by^2-1)}$. 
We assume always that $a,b \in k^*$. 

\begin{prop}\label{trivial} The field of constants of $k_{a,b}$ over $k$ is equal to $k$.
\end{prop}

\begin{proof} Clearly $k_{a,b} = k(x)(\sqrt{\frac{1-ax^2}{b}})$. Suppose $f = f_0+f_1\sqrt{\frac{1-ax^2}{b}}$, $f_0,f_1\in k(x)$, is algebraic over $k$. Then $\overline{f} = f_0-f_1\sqrt{\frac{1-ax^2}{b}}$ is also algebraic over $k$. Consequently, $f_0 = (f+\overline{f})/2$ and $f_0^2-f_1^2(\frac{1-ax^2}{b}) = f\overline{f}$ are algebraic over $k$. It follows that $f_1^2(\frac{1-ax^2}{b})$ is algebraic over $k$, i.e., $f_1=0$, and $f_0 \in k$.
\end{proof}

For $a,b \in k^*$, $(\frac{a,b}{k})$ denotes the quaternion algebra over $k$, i.e., the 4-dimension central simple algebra over $k$ generated by $i,j$ subject to $i^2=a$, $j^2=b$, $ji=-ij$. We identify quaternion algebras over $k$ which are isomorphic as $k$-algebras, equivalently, are equal as elements of the Brauer group of $k$.

\begin{prop} \label{prop_6} The following are equivalent:

(1) $(\frac{a,b}{k})=1$ (i.e., $(\frac{a,b}{k})$ splits over $k$).

(2) $\langle 1,-a\rangle \otimes \langle 1, -b\rangle \sim 0$ over $k$.

(3) $1 \in D_{k}\langle a,b\rangle$.

(4) The conic $ax^2+by^2=1$ has a rational point.

(5) $k_{a,b}$ is purely transcendental over $k$.
\end{prop}

\begin{proof} The equivalence of (1), (2), (3) and (4) is well-known from quadratic form theory.
If $(p,q)$ is a rational point of $ax^2+by^2=1$ then $k_{a,b} = k(z)$ where $z := \frac{y-q}{x-p}$. Conversely, if $k_{a,b} = k(z)$ then, choosing $f(z),g(z),h(z) \in k[z]$ so that $x= \frac{f(z)}{h(z)}$, $y=\frac{g(z)}{h(z)}$, and choosing $r\in k$ so that $h(r)\ne 0$, one sees that $(\frac{f(r)}{h(r)}, \frac{g(r)}{h(r)})$ is a rational point of $ax^2+by^2=1$. Note: This argument fails if $|k|<\infty$, but the conclusion continues to hold even in this case, since $|k|<\infty$ $\Rightarrow$ the quadratic form $\langle a,b\rangle$ is $k$-universal.
\end{proof}

From the definition of $k_{a,b}$ it is clear that $1\in D_{k_{a,b}}\langle a,b\rangle$, so $(\frac{a,b}{k})$ splits over $k_{a,b}$. Of course, $1$ also splits over $k_{a,b}$ (since it splits over $k$). Conversely one has the following:

\begin{prop}[E. Witt] \label{Witt1} The only quaternion algebras defined over $k$ which split over $k_{a,b}$ are $(\frac{a,b}{k})$ and $1$.
\end{prop}

\begin{proof} See \cite[Satz, Page 465]{Witt} or \cite[Lemma 4.4]{k}.
\end{proof}

We write $K \cong_{k} L$ to indicate that the field extensions $K,L$ of $k$ are $k$-isomorphic.

\begin{prop}[E. Witt] \label{Witt2} The following are equivalent:

(1) $(\frac{a,b}{k}) = (\frac{c,d}{k})$.

(2) $k_{a,b} \cong_{k} k_{c,d}$.
\end{prop}

\begin{proof} Then implication (1) $\Rightarrow$ (2) is \cite[Satz, page 464]{Witt}. The implication (2) $\Rightarrow$ (1) is immediate from the Proposition \ref{Witt1}.
\end{proof}


We will need to know which orderings of $k$ extend to $k_{a,b}$.

\begin{lemma}\label{order extension} An ordering $<$ on $k$ extends to an ordering on $k_{a,b}$ iff at least one of $a,b$ is positive at $<$.
\end{lemma}

\begin{proof} One way is clear. If $<$ extends to $k_{a,b}$ then, in $k_{a,b}$, $ax^2+by^2=1>0$ so at least one of $a,b$ must be positive. Conversely, suppose at least one if $a,b$ is positive. Fix a real closure $R$ of $(k,<)$. Clearly $\exists$ $\overline{x}, \overline{y} \in R$ satisfying $a\overline{x}^2+b\overline{y}^2=1$. Then $(\frac{a,b}{k})$ splits over $k(\overline{x},\overline{y})$ and hence also over $R$, so $k_{a,b} \hookrightarrow R_{a,b} \equiv R(t)$. Any one of the infinitely many orderings of $R(t)$ extends the ordering $<$.
\end{proof}

\section{application to function fields of conics}

Let $K$ be a function field in one variable (i.e. a function field $K$ over $k$ satisfying $\operatorname{trdeg}(K:k)=1$) over a global field $k$. (We do not assume that $k$ is the field of constants of $K$ over $k$.) 
We consider the set $\nu_K$ consisting of all non-trivial Abhyankar valuations of $K$ over $k$ such that $K_v$ is not finite of characteristic $2$. 
Thus $$\nu_K = \nu_{K,0}\cup \nu_{K,1} \cup \nu_{K,2} \cup \nu_{K,3} \cup \nu_{K,4} \text{ (disjoint union)}$$ where
$\nu_{K,0}$ is the set of 
valuations $v$ of $K$ such that $\Gamma_v = \mathbb{Z}$ and $K_v$ is a number field, $\nu_{K,1}$ is the set of valuations $v$ on $K$ such that $\Gamma_v =\mathbb{Z}$ and $K_v$ is a global field of characteristic $p\ne 0,2$, $\nu_{K,2}$ is the set of valuations $v$ on $K$ such that $\Gamma_v=\mathbb{Z}$ and $K_v$ is a global field of characteristic $2$, $\nu_{K,3}$ is the set of valuations $v$ on $K$ such that $\Gamma_v = \mathbb{Z} \times \mathbb{Z}$, $K_v$ is a finite field of characteristic $\ne 2$ and $-1 \notin K_v^{*2}$ and $\nu_{K,4}$ is the set of valuations $v$ on $K$ such that $\Gamma_v = \mathbb{Z} \times \mathbb{Z}$, $K_v$ is a finite field of characteristic $\ne 2$ and $-1 \in K_v^{*2}$.

Of course, some of the sets $\nu_{K,i}$ may be empty. Specifically, if $\operatorname{char}(K) \notin \{ 0,2\}$ then $\nu_{K,i} = \emptyset$ for $i \in \{ 0,2\}$ and if $\operatorname{char}(K) = 2$ then  $\nu_{K,i} = \emptyset$ for $i \in \{ 0,1,3, 4\}$.

We will need the following two results from \cite{gm2016}:

\begin{cor}[{\cite[Corollary 8.1]{gm2016}}] \label{one variable} Suppose $K$, $L$ are function fields in one variable over global fields which are Witt equivalent via a hyperfield isomorphism $\alpha : Q(K) \rightarrow Q(L)$. Then for each $i\in \{ 0,1,2,3,4\}$ there is a uniquely defined bijection  between $\nu_{K,i}$ and $\nu_{L,i}$ such that, if $v \leftrightarrow w$ under this bijection, then $\alpha$ maps $(1+M_v)K^{*2}/K^{*2}$ onto $(1+M_w)L^{*2}/L^{*2}$ and $U_vK^{*2}/K^{*2}$ onto $U_wL^{*2}/L^{*2}$ for $i \in \{0,1,2,3\}$ and such that $\alpha$ maps $(1+M_v)K^{*2}/K^{*2}$ onto $(1+M_w)L^{*2}/L^{*2}$ for $i=4$.
\end{cor}

\begin{cor}[{\cite[Corollary 8.2]{gm2016}}] \label{rational points} Let $K \sim L$ be function fields in one variable over global fields $k,\ell$ respectively,  with fields of constants $k$ and $\ell$ respectively. Then $k\sim \ell$ except possibly in the case where $k,\ell$ are both number fields.  In the latter case if there exists $v \in \nu_{K,0}$ with $K_v =k$ and $w \in \nu_{L,0}$ with $L_w= \ell$. Then $k \sim \ell$.
\end{cor}

\begin{rem} \label{rational points remark} The last assertion of Corollary \ref{rational points} applies, in particular, in the case where $K=k(x)$, $L=\ell(x)$, where $k,\ell$ are number fields. \end{rem}

Let $k$ be a number field. Every ordering of $k$ is archimedean, so that it corresponds to a real embedding $k \hookrightarrow \mathbb{R}$.
Denote by $r_1$ the number of real embeddings of $k$ and by $r_2$ the number of conjugate pairs of complex embeddings of $k$. Then $[k:\mathbb{Q}] = r_1+2r_2$. Furthermore, let 
$$V_k := \{ r\in k^* : (r) = \frak{a}^2 \text{ for some fractional ideal } \frak{a} \text{ of } k \}.$$ 
$V_k$  is a subgroup of $k^*$ and $k^{*2} \subseteq V_k$.  We will need the following result, which is a version of \cite[Theorem 8.6]{gm2016}:

\begin{theorem} \label{genus zero case} Suppose $K$ and  $L$ are function fields of genus zero curves over number fields with fields of constants $k$ and $\ell$ respectively, and $\alpha : Q(K) \rightarrow Q(L)$ is a hyperfield isomorphism. Then

(1) $r \in k^*/k^{*2}$ iff $\alpha(r) \in \ell^*/\ell^{*2}$.

(2) $\alpha$ induces a bijection  between orderings $P$ of $k$ which extend to $K$ and orderings $Q$ of $\ell$ which extend to $L$ via $P \leftrightarrow Q$ iff $\alpha$ maps $P^*/k^{*2}$ to $Q^*/\ell^{*2}$.

(3) $\alpha$ maps $V_k/k^{*2}$ to $V_{\ell}/\ell^{*2}$.

(4) $[k:\mathbb{Q}] = [\ell:\mathbb{Q}]$.

(5) $K$ is purely transcendental over $k$ iff $L$ is purely transcendental over $\ell$. In this case, the map $r \mapsto \alpha(r)$ defines a hyperfield isomorphism between $Q(k)$ and $Q(\ell)$, and the $2$-ranks of the ideal class groups of $k$ and $\ell$ are equal.
\end{theorem}

First part the proof follows the same line of reasoning as the proof of \cite[Theorem 8.6]{gm2016}, but we shall provide it here for the sake of the completeness of the exposition. We will need two lemmas:

\begin{lemma}[{\cite[Lemma 8.4]{gm2016}}] \label{2 rank} The $2$-rank of the group $V_k/k^{*2}$ is $r_1+r_2+t$ where $t$ is the $2$-rank of the ideal class group of $k$.
\end{lemma}

\begin{lemma}[{\cite[Lemma 8.5]{gm2016}}] \label{unramified extension} Suppose 
$v$ is a discrete rank 1 valuation on $k$, and $a,b\in k^*$.  There exists an Abhyankar extension of $v$ to $k_{a,b}$ such that $v({k_{a,b}}^*) = v(k^*)$.
\end{lemma}

We now proceed to the proof of Theorem \ref{genus zero case}.

\begin{proof} By our hypothesis, $K = k_{a,b}$ for some $a,b \in k^*$ and $L = \ell_{c,d}$ for some $c,d \in \ell^*$. If $r\in k^*/k^{*2}$, then $r \in U_vK^{*2}/K^{*2}$ for all $v\in \nu_{K,0}$. This implies $\alpha(r) \in U_wL^{*2}/L^{*2}$ for all $w\in \nu_{L,0}$. Since $\frac{\ell[x,y]}{(cx^2+dy^2-1)}$ is a principal ideal domain with unit group $\ell^*$ this implies, in turn, that $\alpha(r) \in \ell^*/\ell^{*2}$.  Thus $\alpha$ induces a group isomorphism between $k^*/k^{*2}$ and $\ell^*/\ell^{*2}$. This proves (1). Suppose $P$ is an ordering of $k$ which extends to an ordering $P_1$ of $K$. Since $\alpha$ is a hyperfield isomorphism, there exists a unique ordering $Q_1$ of $L$ such that $\alpha(P_1^*/K^{*2}) = Q_1^*/L^{*2}$. Denote by $Q$ the restriction of $Q_1$ to $\ell$. Clearly $\alpha(P^*/k^{*2}) = Q^*/\ell^{*2}$. This proves (2).
 Lemma \ref{unramified extension} implies that $$V_k/k^{*2} =\{ r\in k^*/k^{*2} : r \in U_vK^{*2}/K^{*2} \ \forall v \in \nu_{K,1}\cup \nu_{K,2}\},$$ so (3) is clear.  Observe that if $v \leftrightarrow w$, $v\in \nu_{K,0}$, $w \in \nu_{L,0}$, the diagram
$$(*) \left. \xymatrix{
K_v^*/K_v^{*2} \ar[r] &  L_w^*/L_w^{*2} \\
k^*/k^{*2} \ar[u] \ar[r] & \ell^*/\ell^{*2} \ar[u]
}\right.$$
is commutative. The vertical arrows are the maps induced by the field embeddings $k \hookrightarrow K_v$, $\ell \hookrightarrow L_w$. Since the top arrow in ($*$) defines a hyperfield isomorphism between $Q(K_v)$ and $Q(L_w)$ we know that $[K_v:\mathbb{Q}] = [L_w:\mathbb{Q}]$. Choose $v,w$ with $[K_v:\mathbb{Q}] = [L_w:\mathbb{Q}]$ minimal. If $(\frac{a,b}{k})$ splits over $k$ then $K_v=k$ and $[K_v:\mathbb{Q}] = [k:\mathbb{Q}]$. If $(\frac{a,b}{k})$ is non-split over $k$ then $(\frac{a,b}{k})$ splits over a quadratic extension, so $[K_v:k] = 2$, i.e., $[K_v:\mathbb{Q}] = 2[k:\mathbb{Q}]$. Thus we see that either $[k:\mathbb{Q}]= [\ell:\mathbb{Q}]$, $2[k:\mathbb{Q}] = [\ell: \mathbb{Q}]$, or $[k:\mathbb{Q}]=2[\ell:\mathbb{Q}]$. Suppose $2[k:\mathbb{Q}] = [\ell: \mathbb{Q}]$, i.e., $[K_v:\mathbb{Q}] =2$, and $L_w=\ell$. Then the left vertical arrow in ($*$) has a non-trivial kernel, but the right vertical arrow in ($*$) is an isomorphism, a contradiction. Thus $2[k:\mathbb{Q}] = [\ell: \mathbb{Q}]$ is impossible. A similar argument shows that $[k:\mathbb{Q}]=2[\ell:\mathbb{Q}]$ is impossible. This proves (4).  If $(\frac{a,b}{k})$ is non-split and $(\frac{c,d}{\ell})$ is split, then  $[K_v:k] = 2$ and $L_w=\ell$ so $2[k:\mathbb{Q}] = [\ell:\mathbb{Q}]$, contradicting what was proved in (4).
This proves the first statement of (5). If $(\frac{a,b}{k})$ and $(\frac{c,d}{\ell})$ are both split then $k=K_v$, $\ell = L_w$. In this case it follows from the commutativity of ($*$) and what was already proved that the map $r \mapsto \alpha(r)$ defines a hyperfield isomorphism between $Q(k)$ and $Q(\ell)$. Since it is well-known that $r_1$ and $r_2$ are invariant under Witt equivalence, the last assertion of (5) is immediate now, from (3) and Lemma \ref{2 rank}.  This completes the proof.

\end{proof}

\noindent
\bf Question 1: \rm Is it true that every hyperfield isomorphism $\alpha : Q(k_{a,b}) \rightarrow Q(\ell_{c,d})$ induces a hyperfield isomorphism between $Q(k)$ and $Q(\ell)$? 
I.e., is the hypothesis that $(\frac{a,b}{k})$ and $(\frac{c,d}{\ell})$ are split really necessary? 


\medskip

\noindent
\bf Question 2: \rm For a given number field $k$, are there  infinitely many Witt inequivalent fields of the form $k_{a,b}$, $a,b \in k^*$? All we are able to prove in this regard is the following:

\begin{theorem} Let $k$ be a number field, $r = $ the number of orderings of $k$, $w= $ the number of Witt inequivalent fields of the form $k_{a,b}$, $a,b \in k^*$. Then
$$w \ge \begin{cases} 2 \text{ if } -1 \in D_k\langle 1,1\rangle, \\ 3 \text{ if } -1 \notin D_k\langle 1,1\rangle, \ k \text{ is not formally real},\\ r+3 \text{ if } k \text{ is formally real} \end{cases}.$$
\end{theorem}

\begin{proof} For a prime $\frak{p}$ of $k$ (finite or infinite), denote by $\hat{k}_{\frak{p}}$ the completion of $k$ at $\frak{p}$.

Case 1: $-1 \in D_k\langle 1,1\rangle$. Fix $a,b \in k^*$ so that $(\frac{a,b}{k})$ does not split over $k$. E.g., fix some finite prime $\frak{p}$ of $k$ and choose $a,b \in k^*$ so that $a \notin \hat{k}^{*2}_{\frak{p}}$ and $b \notin D_{\hat{k}_{\frak{p}}}\langle 1, - a\rangle$. Theorem \ref{genus zero case} (5) implies that $k_{a,b} \not\sim k_{1,1}$.

Case 2: $-1 \notin D_k\langle 1,1\rangle$, $k$ not formally real. By hypothesis, $(\frac{-1,-1}{k})$ is not split over $k$. By the Hasse norm theorem there exists a finite prime $\frak{p}$ such that $(\frac{-1,-1}{k})$ is not split over $\hat{k}_{\frak{p}}$. By Hilbert reciprocity there exists a finite prime $\frak{q} \ne \frak{p}$ such that $(\frac{-1,-1}{k})$ is not split over $\hat{k}_{\frak{q}}$. Fix $b\in k^*$ so that $b$ is sufficiently close to $-1$ at $\frak{p}$ and $b$ is sufficiently close to $1$ at $\frak{q}$. Then $b$ is minus a square in $\hat{k}_{\frak{p}}$ and $b$ is a square in $\hat{k}_{\frak{q}}$. If there exists a hyperfield isomorphism $\alpha : Q(k_{-1,-1}) \rightarrow Q(k_{b,-1})$ then, since $(\frac{-1,-1}{k})$ splits over $k_{-1,-1}$ and $\alpha(-1)=-1$, it follows that $(\frac{-1,-1}{k})$ splits over $k_{b,-1}$. According to Proposition \ref{Witt1}, this implies $(\frac{-1,-1}{k}) =1$ or $(\frac{-1,-1}{k}) = (\frac{b,-1}{k})$, i.e., $(\frac{-1,-1}{k}) =1$ or $(\frac{-b,-1}{k}) = 1$. Since  $(\frac{-1,-1}{k})$ does not split over $\hat{k}_{\frak{p}}$ and $(\frac{-b,-1}{k})$ does not split over $\hat{k}_{\frak{q}}$ this is a contradiction. Thus $k_{-1,-1} \not\sim k_{b,-1}$. Since $(\frac{-1,-1}{k}) \ne 1$ and $(\frac{b,-1}{k}) \ne 1$ one also has that $k_{-1,-1} \not\sim k_{1,1}$ and $k_{b,-1} \not\sim k_{1,1}$.

Case 3: Suppose $k$ is formally real. 
For each integer $0 \le i \le r$ choose $a_i \in k^*$ so that $a_i>0$ for exactly $i$ of the orderings of $k$. Without loss of generality, $a_0 = -1$, $a_r=1$. By Lemma \ref{order extension} exactly $i$ of the orderings of $k$ extend to $k_{a_i,-1}$, so $k_{a_i,-1} \not\sim k_{a_j,-1}$ if $i\ne j$ by Theorem \ref{genus zero case} (2). This proves $w \ge r+1$. Fix $0\le i \le r$ so that $r-i$ is odd (e.g., take $i = r-1$). By Hilbert reciprocity there exists a finite prime $\frak{p}$ of $k$ such that $(\frac{a_i,-1}{k})$ does not split over $\hat{k}_{\frak{p}}$. Pick $\overline{a}_i \in k^*$ sufficiently close to $1$ at $\frak{p}$ and sufficiently close to $a_i$ at each ordering of $k$. Then $\overline{a}_i$ is a square in $\hat{k}_{\frak{p}}$ and has the same sign as $a_i$ at each ordering. By Hilbert reciprocity there exists a finite prime $\frak{q} \ne \frak{p}$ so that $(\frac{\overline{a}_{i},-1}{k})$ does not split over $\hat{k}_{\frak{q}}$. Pick $b_r\in k^*$ so that $b_r$ is close to $a_r (=1)$ at each ordering of $k$ and such that $b_r$ is close to $a_i$ at $\frak{p}$. Then $k_{b_r,-1} \not\sim k_{a_r,-1}$, by Theorem \ref{genus zero case} (5) (because $(\frac{a_r,-1}{k})$ splits over $k$ but $(\frac{b_r,-1}{k})$ does not split over $\hat{k}_{\frak{p}}$). Also, all orderings of $k$ extend to $k_{b_r,-1}$, so $k_{b_r,-1} \not\sim k_{a_j,-1}$ for $0\le j <r$, by Theorem \ref{genus zero case} (2). Finally, pick $b_0 \in k^*$ so that $b_0$ is close to $a_0 (=-1)$ at each ordering of $k$, $b_0$ is close to $a_i$ at $\frak{p}$ and $b_0$ is close to $-\overline{a}_i$ at $\frak{q}$. A similar argument to that used in Case 2 shows that $k_{b_0,-1} \not\sim k_{-1,-1}$. Theorem \ref{genus zero case} (2) shows that $k_{b_0,-1} \not\sim k_{a_j,-1}$ for $0<j\le r$ and $k_{b_0,-1} \not\sim k_{b_r,-1}$.
\end{proof}

\noindent
\bf Question 3: \rm For a fixed integer $n\ge 1$, are there infinitely many Witt inequivalent fields $k_{a,b}$, $k$ a number field, $[k:\mathbb{Q}] = n$, $a,b\in k^*$?

\smallskip

Following the argument of \cite[Corollary 8.8]{gm2016} we are able to partially answer this question in case $n=2$: let $d$ be a square free integer and denote by $N$ the number of prime integers that ramify in $\mathbb{Q}(\sqrt{d})$. This is equal to the number of prime divisors of the discriminant of $\mathbb{Q}(\sqrt{d})$. Recall that the discriminant of $\mathbb{Q}(\sqrt{d})$ is $d$ if  $d \equiv 1 \mod 4$ and $4d$ otherwise.   Then the $2$-rank of the class group of $\mathbb{Q}(\sqrt{d})$ is  $$ \begin{cases} N-2 \text{ if } d>0 \text{ and } d \notin D_{\mathbb{Q}}\langle 1,1\rangle,\\ N-1 \text{ otherwise.} \end{cases}$$  
See \cite[Corollary 18.3]{ch} for the proof. In particular, there are infinitely many possible values for the $2$-rank of the class number for fields of the sort $\mathbb{Q}(\sqrt{d})$, $d \in \mathbb{Q}^*\backslash \mathbb{Q}^{*2}$. Combining this with Theorem \ref{genus zero case} yields:

\begin{cor}\label{infinitely many} There are infinitely many Witt inequivalent fields of the form $k(x)$, $k$ a quadratic extension of $\mathbb{Q}$. 
\end{cor}

Finally, we consider an application of Corollary \ref{one variable} and Theorem \ref{genus zero case} to fields of the form $\mathbb{Q}_{a,b}$.

\begin{prop} \label{prime matching} Suppose $ \alpha : Q(\mathbb{Q}_{a,b}) \rightarrow Q(\mathbb{Q}_{c,d})$ is a hyperfield isomorphism. Then, 
for each prime integer $p$, $\alpha(p) =\pm q$ for some prime integer $q$, and $p=2$ $\Rightarrow$ $q=2$.
\end{prop}


\begin{proof} Let $K = \mathbb{Q}_{a,b}$, $L= \mathbb{Q}_{c,d}$. Theorem \ref{genus zero case} shows that $r\mapsto \alpha(r)$ defines a group automorphism of $\mathbb{Q}^*/\mathbb{Q}^{*2}$. For $r\in \mathbb{Q}^*/\mathbb{Q}^{*2}$, define
$$\mathcal{S}_K(r) := \{ v\in \nu_{K,1}\cup \nu_{K,2} : r \notin U_vK^{*2}\}.$$
Note that $\mathcal{S}_K(\pm 1)$ is the empty set.
Let $p$ be a prime. The set $\mathcal{S}_K(\pm p)$  is non-empty (by Lemma \ref{unramified extension}) and is minimal among all non-empty $\mathcal{S}_K(r)$. It follows that $\mathcal{S}_L(\alpha(p))$ is non-empty and minimal among all non-empty $\mathcal{S}_L(s)$, $s\in \mathbb{Q}^*/\mathbb{Q}^{*2}$, so $\alpha(p) = \pm q$, for some prime $q$. Note also that if $p=2$, then $\mathcal{S}_K(p)$ is a subset of $\nu_{K,2}$. Then $\mathcal{S}_L(\pm q)$ is a subset of $\nu_{L,2}$, so $q=2$.
\end{proof}

Function fields of conics defined over local fields have been investigated by the authors in their earlier work \cite{gm2016a}. There, the following versions of Corollaries \ref{one variable} and \ref{rational points} are established for local fields:

\begin{theorem}[{\cite[Theorem 3.5]{gm2016a}}] \label{local one variable} Suppose $K$, $L$ are function fields in one variable over local fields of characteristic $\ne 2$ which are Witt equivalent via a hyperfield isomorphism $\alpha : Q(K) \rightarrow Q(L)$. Then for each $i\in \{ 0,1,2,3\}$ there is a uniquely defined bijection  between $\mu_{K,i}$ and $\mu_{L,i}$ such that, if $v \leftrightarrow w$ under this bijection, then $\alpha$ maps $(1+M_v)K^{*2}/K^{*2}$ onto $(1+M_w)L^{*2}/L^{*2}$ and $U_vK^{*2}/K^{*2}$ onto $U_wL^{*2}/L^{*2}$ for $i \in \{0,1,2\}$ and such that $\alpha$ maps $(1+M_v)K^{*2}/K^{*2}$ onto $(1+M_w)L^{*2}/L^{*2}$ for $i=3$.
\end{theorem}

\begin{theorem}[{\cite[Theorem 3.6]{gm2016a}}] \label{local rational points} Let $K \sim L$ be function fields in one variable over local fields $k$ and $\ell$ respectively,  with fields of constants $k$ and $\ell$ respectively. Then $k\sim \ell$ except possibly when $k,\ell$ are both dyadic local fields.  In the latter case if there exists $v \in \mu_{K,0}$ with $K_v =k$ and $w \in \mu_{L,0}$ with $L_w= \ell$ then $k \sim \ell$.
\end{theorem}

In view of the abovementioned results, we are able to slightly extend the results of Proposition \ref{prime matching} to the local case:

\begin{theorem} Suppose $k, \ell$ are local fields of characteristic $\ne 2$, $a,b \in k^*$, $c,d \in \ell^*$. Then $k_{a,b} \sim \ell_{c,d}$ $\Rightarrow$ $k \sim \ell$.
\end{theorem}

\begin{proof} By Theorem \ref{local rational points} it suffices to deal with the case where $k,\ell$ are both dyadic. Let $\alpha : Q(k_{a,b}) \rightarrow Q(\ell_{c,d})$ be some hyperfield isomorphism. Making use of the bijection between $\mu_{K,0}$ and $\mu_{L,0}$ induced by $\alpha$, and arguing as in the proof of Theorem \ref{genus zero case} we see that $\alpha$ induces an isomorphism between $k^*/k^{*2}$ and $\ell^*/\ell^{*2}$. This implies $[k:\hat{\mathbb{Q}}_2] = [\ell:\hat{\mathbb{Q}}_2]$. Thus $k \sim \ell$ iff $k,\ell$ have the same level. The level of a dyadic local field is $1,2$ or $4$. If $k$ has level $4$ then $[k:\hat{\mathbb{Q}}_2] = [\ell:\hat{\mathbb{Q}}_2]$ is odd so $\ell$ has level $4$. If $k$ has level $1$ then $k_{a,b}$ and consequently also $\ell_{c,d}$ has level $1$. Since $\ell$ is algebraically closed in $\ell_{c,d}$ this implies $\ell$ has level $1$.
\end{proof}

\begin{theorem} Suppose $k$ is a local field of characteristic $\ne 2$, $a,b,c,d \in k^*$. Then $k_{a,b} \sim k_{c,d}$ $\Rightarrow$ $(\frac{a,b}{k}) = (\frac{c,d}{k})$ except possibly in the case when $k$ is $p$-adic of level $1$, for some odd prime $p$.
\end{theorem}

\begin{proof} If $k=\mathbb{C}$ there is only one quaternion algebra and the result is obvious. Otherwise, there are two quaternion algebras, one split and one non-split. Suppose $K = k_{a,b}$, $L = k_{c,d}$, $(\frac{a,b}{k})$ split, $(\frac{c,d}{k})$ non-split. Suppose $k = \mathbb{R}$. Then $K$ is formally real and $L$ is non-real (of level $2$), so $K \not\sim L$. Suppose now that $k$ is dyadic. Suppose $K\sim L$. Fix a hyperfield isomorphism $\alpha : Q(K) \rightarrow Q(L)$. Fix $v \in \mu_{K,0}$ with $K_v=k$ and let $w$ be the corresponding element of $\mu_{L,0}$. Then $k = K_v \sim L_v$ so $[k:\hat{\mathbb{Q}}_2] = [K_v:\hat{\mathbb{Q}}_2] = [L_w:\hat{\mathbb{Q}}_2]$. Since $k \subseteq L_w$, this forces $L_w = k$, i.e., $(\frac{c,d}{k})$ splits, a contradiction. Suppose now that $k$ is $p$-adic, $p \ne 2$, and $k$ has level $2$. Since $k$ has level $2$, we may assume $c=\pi$, where $v_0(\pi) = 1$, and $d=-1$. Thus $\exists$ $f,g \in L^*$ such that $\pi = f^2+g^2$. For each $w \in \mu_{L,1}$, $w(\pi) = v_0(\pi)=1$ is odd, so $w(f^2) = w(g^2) <w(f^2+g^2)$, i.e., $-1$ is a square in $L_w$, for all $w \in \mu_{L,1}$.  Suppose $K \sim L$. Fix a hyperfield isomorphism $\alpha : Q(K) \rightarrow Q(L)$. Then the induced one-to-one correspondence $v \leftrightarrow w$ between $\mu_{K,1}$ and $\mu_{L,1}$ and the induced hyperfield isomorphisms $Q(K_v) \rightarrow Q(L_w)$ imply $-1$ is a square in $K_v$ for all $v \in \mu_{K,1}$. Define one particular such $v$ as follows: Since $(\frac{a,b}{k})$ splits, $K = k(x)$. Extend $v_0$ to $K$ by defining $v(\sum_{i=0}^n a_ix^i) = \min\{ v_0(a_i) : i\in \{ 0,\dots,n\}\}$. Clearly $v\in \mu_{K,1}$ and $K_v = k_{v_0}(x)$. But then $K_v$ and $k_{v_0}$ both have the same level, contradicting the fact that $K_v$ has level $1$ and $k_{v_0}$ has level $2$.
\end{proof}

\end{document}